\documentclass{amsart}
\usepackage{graphics}
\usepackage{amsfonts,amsmath}
\usepackage{amsmath, amsfonts,amssymb,amsopn,amscd,amsthm}
 \usepackage{hyperref}
 \usepackage{mathrsfs}
 \hypersetup{colorlinks=true,linkcolor=blue,citecolor=blue,linktoc=page}

\begin{document}

 \newtheorem{thm}{Theorem}[section]
 \newtheorem{cor}[thm]{Corollary}
 \newtheorem{lem}[thm]{Lemma}{\rm}
 \newtheorem{prop}[thm]{Proposition}

 \newtheorem{defn}[thm]{Definition}{\rm}
 \newtheorem{assumption}[thm]{Assumption}
 \newtheorem{rem}[thm]{Remark}
 \newtheorem{ex}{Example}
\numberwithin{equation}{section}
\def\la{\langle}
\def\ra{\rangle}
\def\glexe{\leq_{gl}\,}
\def\glex{<_{gl}\,}
\def\e{{\rm e}}

\def\x{\mathbf{x}}
\def\P{\mathbf{P}}
\def\h{\mathbf{h}}
\def\by{\mathbf{y}}
\def\bz{\mathbf{z}}
\def\F{\mathcal{F}}
\def\R{\mathbb{R}}
\def\T{\mathbf{T}}
\def\N{\mathbb{N}}
\def\D{\mathbf{D}}
\def\V{\mathbf{V}}
\def\U{\mathbf{U}}
\def\K{\mathbf{K}}
\def\Q{\mathbf{Q}}
\def\H{\mathbf{H}}
\def\M{\mathbf{M}}
\def\oM{\overline{\mathbf{M}}}
\def\O{\mathbf{O}}
\def\C{\mathbb{C}}
\def\P{\mathbf{P}}
\def\Z{\mathbb{Z}}
\def\H{\mathcal{H}}
\def\A{\mathbf{A}}
\def\V{\mathbf{V}}
\def\AA{\overline{\mathbf{A}}}
\def\B{\mathbf{B}}
\def\c{\mathbf{C}}
\def\L{\mathcal{L}}
\def\bS{\mathbf{S}}
\def\H{\mathbf{H}}
\def\I{\mathbf{I}}
\def\Y{\mathbf{Y}}
\def\X{\mathbf{X}}
\def\G{\mathbf{G}}
\def\f{\mathbf{f}}
\def\z{\mathbf{z}}
\def\v{\mathbf{v}}
\def\y{\mathbf{y}}
\def\d{\hat{d}}
\def\bx{\mathbf{x}}
\def\bI{\mathbf{I}}
\def\y{\mathbf{y}}
\def\g{\mathbf{g}}
\def\w{\mathbf{w}}
\def\b{\mathbf{b}}
\def\a{\mathbf{a}}
\def\u{\mathbf{u}}
\def\q{\mathbf{q}}
\def\e{\mathbf{e}}
\def\s{\mathcal{S}}
\def\cc{\mathcal{C}}
\def\co{{\rm co}\,}
\def\tg{\tilde{g}}
\def\tx{\tilde{\x}}
\def\tg{\tilde{g}}
\def\tA{\tilde{\A}}

\def\supmu{{\rm supp}\,\mu}
\def\supp{{\rm supp}\,}
\def\cd{\mathcal{C}_d}
\def\cok{\mathcal{C}_{\K}}
\def\cop{COP}
\def\vol{{\rm vol}\,}
\def\om{\mathbf{\Omega}}
\def\blue{\color{blue}}
\def\red{\color{red}}
\def\blambda{\boldsymbol{\lambda}}
\def\balpha{\boldsymbol{\alpha}}
\def\dis{\displaystyle}
\def\1{\boldsymbol{1}}
\def\va{\vert\balpha\vert}
\title{An extension of the Mean Value Theorem}
\author{Jean B. Lasserre}
\address{LAAS-CNRS and Toulouse School of Economics (TSE)\\
LAAS, 7 avenue du Colonel Roche\\
31077 Toulouse C\'edex 4, France\\
Tel: +33561336415}
\email{lasserre@laas.fr}

\date{}

\begin{abstract}
Let $(\om,\mu)$ be a measure space with $\om\subset\R^d$ and 
$\mu$ a finite measure on $\om$.
We  provide an extension of the  Mean Value Theorem (MVT) in the form
 $\int_\om fd\mu =\mu(\om)(a\,f(\x_0)+(1-a)\,f(\x_1))$, 
 with $a\in [0,1]$ and $\x_0,\x_1\in \om$. It is valid for non compact sets $\om$ and 
 $f$ is only required to be integrable with respect to $\mu$.
 It also contains as a special case the MVT in the form
 $\int f\,d\mu = \mu(\om) f(\x_0)$ for some $\x_0\in\om$,
 valid for compact  connected set $\om$ and continuous $f$.
 It is a direct consequence of Richter's theorem which in turn is  a non
 trivial (overlooked) generalization of Tchakaloff's theorem, and even published earlier.
\end{abstract}
\maketitle
\section{Introduction}

The Mean Value Theorem (MVT) is quite fundamental and widely known and 
covered in most textbooks in real analysis.  It states that with a compact connected set $\om\subset\R^d$,
a continuous function $f:\om\to\R$, and a finite Borel measure $\mu$ on $\om$, there exists $\x_0\in\om$ such that 
\begin{equation}
\label{mvt-basic}
\int f\,d\mu =f(\x_0)\,\mu(\om)\,.
\end{equation}
\begin{proof}
$\int_\om f\,d\mu\in [f_*\,\mu(\om)\,,\,f^*\,\mu(\om)]$, where
 $f_*=\min_{\x\in\om}f(\x)$, $f^*=\max_{\x\in\om}f(\x)$. In addition, as $\om$ is connected and $f$
 is continuous, $f(\om)=[f_*\,,\,f^*]$. Suppose not, i.e., the exists 
 $a\in [f_*\,,\,f^*]$ such that $f(\x)\neq a$ for all $\x\in\om$. Then
 by continuity of $f$, the sets $A:=f^{-1}([f_*,a])$ and $B:=f^{-1}([a,f^*])$ are closed and disjoint,
 and $\om\subset A\cup B$. But $\om\cap A\neq\emptyset$ and $\om\cap B\neq\emptyset$
 implies that $\om$ is not connected, a contradiction. Hence 
 $f(\om)=[f_*\,,\,f^*]$, which in turn implies that
 there exists $\x^*\in\om$ such that $\int f\,d\mu=f(\x^*)\mu(\om)$. 
 \end{proof}
 
This note is concerned with a non-trivial extension
of the MVT which follows from Richter's theorem, a result in real analysis
that has been overlooked in the literature on the Moment problem.
Indeed, although our contribution is a direct and
easy consequence of Richter's theorem,  
to the best of our knowledge it has not appeared in the literature, 
at least in this form.

Therefore, considering the importance of the MVT and its restrictions of compactness and continuity
to be applicable, we think that in view of its simplicity and generality,
its extension is potentially useful in many settings
where the classical MVT fails. As we next see, the extension is indeed valid in a quite
general context. 

This note is organized as follows. We first state Richter's theorem and Tchakaloff's theorem in real analysis on the moment problem and
provide historical details mostly found in \cite{math-anal} and \cite{schmudgen} where the fact that Richter's theorem has indeed been overlooked
is also mentioned.  Then we state our main result on the extension of the MVT is  a context that is 
far more general than its standard version. An elementary example is provided to illustrate 
the result.
\section{Main result}

\vspace{.2cm}

\subsection*{Richter's theorem}

Let $(\om,\mu)$ be a measure space  and denote by $L^1(\om,\mu)$ the Lebesgue space of real integrable functions with respect to $\mu$. Denote also by $M_+(\om)$ the space of Radon measures on $\om$, and by $\delta_\x$ the Dirac measure at the point $\x\in\om$.

In its simplest and most accessible version taken from \cite{schmudgen}, Richter's theorem (called Richter-Tchakaloff theorem in \cite{schmudgen})
reads as follows

\begin{thm}(\cite[Theorem 1.24]{schmudgen})
\label{th-Richter}
Suppose that $(\om,\mu)$ is a measure space, $V$ is a finite-dimensional linear subspace of $L^1(\om,\mu)$, and $L^\mu$ denotes the linear functional on $V$ defined by
\[L^\mu(f)=\int f\,d\mu\,,\quad \forall f\in V\,.\]
Then there is a $k$-atomic measure $\nu=\sum_{j=1}^k m_j\,\delta_{\x_i}\in M_+(\om)$, where
$k\leq \mathrm{dim}(V)$, such that $L^\mu=L^\nu$, that is:
\begin{equation}
 \label{th-Richter-1}
 \int f \,d\mu\,=\,\int f\,d\nu\,=\,\sum_{j=1}^k m_j\,f(\x_j)\,,\quad \forall f\,\in\,V\,.
\end{equation}
\end{thm}
Tchakaloff's theorem  \cite{Tchakaloff} also states  \eqref{th-Richter-1}  but for $\om$ compact
and $V=\R[\x]_k$ (the  space of polynomials of degree at most $k$), a much more restrictive setting.

Tchakaloff's theorem is quite useful for moment problems and cubatures in numerical integration.
Indeed,  for instance  if one knows moments 
\[\mu_{\alpha}\,=\,\int_{\om} \x^{\alpha}\,d\mu\,=\,\int_{\om} x_1^{\alpha_1}\cdots x_d^{\alpha_d}\,d\mu
\,,\quad \alpha\in\N^d_n\,\]
up to degree-$n$, of an unknown measure $\mu$ on $\om\subset\R^d$,
then there exists a $k$-atomic measure $\nu$ on $\om$, supported on at most $s\leq {n+d\choose d}$ atoms 
$\x(1),\ldots,\x(s)\in\om$ and with same moments up to degree-$n$. Therefore one may construct \emph{cubatures}
supported on such points with positive weights $\gamma_1,\ldots,\gamma_s$. That is 
given a measurable function $f:\om\to\R$, one  approximates the integral $\int f\,d\mu$ with
\[\int f\,d\nu\,=\, \sum_{j=1}^s \gamma_j\,f(\x(j))\,,\]
with the guarantee that
\[\int p(\x)\,d\mu(\x)\,=\,\sum_{j=1}^s \gamma_j\,p(\x(j))\,,\quad\forall p\in\R[\x]_n\,,\]
where $\R[\x]_n$ is the space of polynomials with total degree up to $n$.

\subsection*{Historical notes} According to Dio and Schm\"udgen \cite[p. 11]{math-anal}, \emph{``The history of Richter's theorem is confusing and intricate and often the corresponding references in the literature are misleading."}
In \cite{math-anal} the authors mention that Rosenbloom \cite[Corollary 38e]{Rosenbloom} proved \eqref{th-Richter-1} 
for vector spaces $V$ of \emph{bounded} measurable functions. Rogosinski \cite[Theorem 1]{Rogosinski}
(submitted about a half year after Richter \cite{Richter}) also proved \eqref{th-Richter-1} for the one-dimensional case
but claims that his proof also works for general measurable spaces.  
They also precise that while Richter's result seems to treat only the one-dimensional case,  a closer look reveals that it covers the general case of measurable functions. Hence Tchakaloff's theorem in 1958 is a special case of Rosenbloom  \cite{Rosenbloom} in 1952, while Rogosinski and Richter proved the general case almost about at the same time.

\subsection*{Our main result}

\begin{prop}
\label{prop-main}
Let $(\om,\mu)$ be a measure space, with  $\mu$  a finite measure on $\om$ with mass $\mu(\om)>0$, and
let $f:\om\to\R$ be integrable with respect to $\mu$. Then 
there exist $\x_0,\x_1\in\om$ and $\lambda\in [0,1]$, such that
\begin{equation}
\label{eq:1}
\tau:=\int_\om f\,d\mu\,=\,
\mu(\om)\,(\lambda f(\x_0)+(1-\lambda)\,f(\x_1))\,,
\end{equation}
that is, $\tau/\mu(\om)$ is a convex combination of $f(\x_0)$ and $f(\x_1)$.
\end{prop}
\begin{proof}
Let $\1$ be the constant function equal to $1$ for all $\x\in\om$. Then
 \[\tau\,=\,\int_\om f\,d\mu\,;\quad\mu(\om)\,=\,\int_\om \1\,d\mu\,.\]
 Both $f$ and $\1$ are integrable w.r.t. $\mu$. Then by Richter's theorem 
 (\cite[Theorem 2.1.1, p. 39]{anastassiou} and \cite[Theorem 1.24, p.23]{schmudgen}),
 there exists an atomic (positive) measure $\nu:=a\,\delta_{\x_0}+b\,\delta_{\x_1}$ with $a,b\geq0$, supported on $2$ points $\x_0,\x_1\in\om$, and such that
 \begin{eqnarray*}
 \tau\,=\,\int_\om f\,d\nu&=&a\,f(\x_0)+b\,f(\x_1)\,;\\
 0\,<\,\mu(\om)\,=\,\int_\om \1\,d\nu&=&a+b\,.
 \end{eqnarray*}
 Then setting $\lambda:=a/\mu(\om)$ yields the desired result \eqref{eq:1}.
\end{proof}
 As the reader can see, the proof is  a direct consequence of Richter  Theorem \ref{th-Richter}.
 The price to pay for the extension \eqref{eq:1} of  \eqref{mvt-basic}
 to integrable functions and arbitrary measure spaces $(\om,\mu)$,
 is  relatively moderate. Indeed the $\mu$-average value of $f$ is now a convex combination of at most two   values of $f$ instead of a single value $f(\x_0)$ in \eqref{mvt-basic}.

  We share the opinion in \cite{schmudgen} that in contrast to (the far more restrictive) Tchakaloff's theorem
  quite cited in the literature on cubatures and the moment problem,
  Richter's theorem had been overlooked. For instance, quoting \cite[p. 41]{schmudgen}
  \emph{``Richter's paper has been ignored in the literature and a number of versions of his result have been reproved even recently."} This may explain why \eqref{eq:1} has not been stated already (at least in this simple form).
  
  \subsection*{Illustrative example}
  We end up this note by a simple illustrative toy example. Let $\om:=[0,1]$ and
 $\mu$ be the Lebesgue measure on $[0,1]$.
 Let $x\mapsto f(x):=1_{[0,1/2]}(x)+2\cdot 1_{(1/2,1]}(x)$. Hence
 \[\int_0^1f\,dx\,=\,1/2+2/2=3/2\,\not\in \mu(\om)\,f([0,1])=\{1,2\}\,.\]
 On the other hand, let $x_0\in [0,1/2]$ and $x_1\in (1/2,1]$, be fixed arbitrary. Then
 \begin{eqnarray*}
 \int_0^1 f\,d\mu\,=\,3/2&=&\mu(\om)\,(f(x_1)+f(x_2))/2\quad\mbox{[as in \eqref{eq:1}]}\\
 &=& \int_{\om} f\,d\nu\quad\mbox{with}\quad \nu=\frac{1}{2}\delta_{x_0}+\frac{1}{2}\delta_{x_1}\,.\end{eqnarray*}
  \section{Conclusion}\label{sec13}
We agree with \cite{schmudgen} that (the important) Richter's theorem has been overlooked in the literature,
which may explain why despite its simplicity and generality,
the above extension of the MVT has not appeared in this form
(at least to the best of our knowledge).

 \subsection*{Declaration and Acknowledgement}
There is no competing interest.
 
The  author  is supported by the Artificial and Natural Intelligence Toulouse Institute,
ANITI IA Cluster - ANR-23-IACL-0002-IACL - 2023.


\begin{thebibliography}{las}
\bibitem{anastassiou}
G.A. Anastassiou, \emph{Moments in Probability and Approximation Theory}, Longman Scientific \& Technical, England, 1993.
\bibitem{math-anal}
P. J. di Dio, K. Schm\"udgen, The multi-dimensional truncated moment problem: The moment cone,
\emph{J. Math. Anal. Appl.} 511, 126066, 2022.
\bibitem{Richter} 
H. Richter, Parameterfreie Absch\"atzung von Erwartungswerten, B1, \emph{Deutsche Ges. Versicherungsmath} 3, pp. 147--161, 1957.
\bibitem{Rogosinski}
W.W. Rogosinski, Moments on non-negative mass, \emph{Proc. R. Soc. Lond. A} 245, pp. 1--27, 1958.
\bibitem{Rosenbloom}
P.C. Rosenbloom, Quelques classes de de probl\`emes extr\'emaux. II, \emph{Bull. Soc. Math. Fr.} 80, pp. 183--215, 1952.
\bibitem{schmudgen}
K. Schm\"udgen, \emph{The Moment Problem}, Springer, 2017
\bibitem{Tchakaloff}
M.V. Tchakaloff, Formules de cubatures m\'ecaniques \`a coefficients non n\'egatifs, \emph{Bull. Sci. Math.} 81, pp. 123--134, 1957.
 \end{thebibliography}
 \end{document}